\documentclass[11pt]{amsart}
\usepackage[margin=1in]{geometry}

\usepackage{amssymb}
\usepackage{amsthm}
\usepackage{amsmath}
\usepackage{mathrsfs}
\usepackage{amsbsy}
\usepackage{bm}
\usepackage{hyperref}
\usepackage{tikz}
\usepackage{array}
\usepackage{enumerate}
\usepackage{xcolor}
\usepackage{bbm}
\usepackage{comment}

\definecolor{lavender}{rgb}{0.4,0,1}

\usepackage[noabbrev,capitalise,nameinlink]{cleveref}

\hypersetup{colorlinks=true, citecolor=lavender, linkcolor=lavender,urlcolor=lavender}

\usepackage{hhline}
\allowdisplaybreaks
\usepackage[noadjust]{cite}

\usepackage{caption}
\usepackage[noabbrev,capitalise,nameinlink]{cleveref}
\crefname{conjecture}{Conjecture}{Conjectures}

\newtheorem{theorem}{Theorem}[section]
\newtheorem{proposition}[theorem]{Proposition}
\newtheorem{corollary}[theorem]{Corollary}

\newtheorem{lemma}[theorem]{Lemma}

\theoremstyle{definition}

\newtheorem{remark}[theorem]{Remark}

\newcommand{\dd}{\boldsymbol{D}}
\newcommand{\rr}{\boldsymbol{r}}
\newcommand{\dfn}[1]{\textcolor{blue}{\emph{#1}}}

\newcommand{\clock}[1]{\protect\tikz [baseline=0pt] {
                  \protect\fill (0ex,.5ex) circle (.3ex);
                  \protect\draw[line width=.1ex] (-0ex, -.5ex)--(-.25ex, .5ex);
                   \protect\draw[line width=.1ex] (.25ex, .5ex)--(0ex, 1.5ex);} \kern .3ex #1}
                   
\newcommand{\counter}[1]{\protect\tikz [baseline=0pt] {
                 \protect\fill (0ex,.5ex) circle (.3ex);
                  \protect\draw[line width=.1ex] (0ex, 1.5ex)--(-.25ex, .5ex);
                   \protect\draw[line width=.1ex] (.25ex, .5ex)--(0ex, -.5ex);}\kern .3ex #1}

\begin{document}

\title[]{Mixing on Generalized Associahedra}
\subjclass[2010]{}

\author[]{William Chang}
\address[]{Department of Applied Mathematics, University of California, Los Angeles, USA}
\email{chang314@g.ucla.edu}

\author[]{Colin Defant}
\address[]{Department of Mathematics, Harvard University, Cambridge, MA 02138, USA}
\email{colindefant@gmail.com}

\author[]{Daniel Frishberg}
\address[]{Department of Computer Science and Software Engineering, California Polytechnic State University, San Luis Obispo, CA, USA}
\email{dfrishbe@calpoly.edu}

\begin{abstract}
Eppstein and Frishberg recently proved that the mixing time for the simple random walk on the $1$-skeleton of the associahedron is $O(n^3\log^3 n)$. We obtain similar rapid mixing results for the simple random walks on the $1$-skeleta of the type-$B$ and type-$D$ associahedra. We adapt Eppstein and Frishberg's technique to obtain the same bound of $O(n^3\log^3 n)$ in type $B$ and a bound of $O(n^{13} \log^2 n)$ in type $D$; in the process, we  establish an expansion bound that is tight up to logarithmic factors in type $B$. 
\end{abstract} 

\maketitle

\section{Introduction}\label{sec:intro}

\subsection{Monte Carlo sampling algorithms}
\label{sec:mcmc}
\emph{Markov chain Monte Carlo} (\emph{MCMC}) algorithms are widely studied for a broad class of sampling problems in which one wishes to sample from a distribution over a large state space. Well-motivated, classic sampling problems include the \emph{Ising} and \emph{Potts} models of magnetism and the \emph{hardcore} model of gases. For a survey of Markov chain mixing times, including discussion of classic problems, see~\cite{lpwbook}. These problems are closely related to  the combinatorial problems of sampling and counting $k$-colorings and independent sets in graphs. Each of these problems involves sampling from a distribution over a state space\textemdash such as the collection of all $k$-colorings or independent sets of an input graph~$G$. These state spaces are typically exponentially large with respect to the input graph. MCMC algorithms\textemdash in which one begins with a non-random state and performs a random walk through the state space\textemdash comprise a simple and natural family of algorithms for these sampling problems.

A central question in the analysis of MCMC sampling algorithms is their convergence rate, or \emph{mixing time}. Although the state space is typically exponential in size, one typically wishes to find a polynomial-time (in terms of $n$, the number of vertices of the input graph) algorithm for sampling. More precisely, MCMC algorithms can be run until an arbitrary stopping point $T$. As $T\rightarrow \infty$, these algorithms can be shown to converge to an exact sample\textemdash but one wishes to bound the time before approximate convergence (see \cref{sec:setup} for precise definitions), within a desired constant tolerance threshold, to the target distribution (called the \emph{stationary} distribution). 

The analysis of mixing times dates back decades~\cite{lpwbook}. A number of papers over the past several years~\cite{logconcaveii,  evoptimal, anarihardcore, chen2021rapid, clv} have established polynomial and in some cases optimal mixing time bounds for a number of MCMC sampling algorithms, using recently developed algebraic machinery~\cite{ko18, alevlau} relying on spectral graph theory and abstract simplicial complexes. See the recent lecture notes by \v{S}tefankovi\v{c} and Vigoda~\cite{vigstef} for an introduction to these techniques.

A classic problem of direct interest for the present paper concerns sampling \emph{triangulations} of a convex polygon. A triangulation is a maximal set of mutually non-crossing \emph{diagonals}\textemdash non-polygon edges whose endpoints are vertices of the polygon. One can define a move known as a \emph{flip}: one chooses a diagonal in the polygon, removes the diagonal, and replaces it with the (unique) other diagonal that avoids creating crossings. The flip move gives rise to a natural Markov chain\textemdash the \emph{flip walk}\textemdash for sampling a uniform $n$-gon triangulation, in which one repeatedly flips a random diagonal until the chain has approximately converged to a uniform sample.

The state space of this Markov chain is known as the \emph{associahedron}. This chain and other triangulation \emph{flip graphs}~\cite{eppflip, randlattice, Stauffer2017} have been studied from a geometric and combinatorial perspective~\cite{eppflip}. Furthermore, the triangulations of an $n$-gon are in bijection with the full binary trees consisting of $n-2$ internal nodes\textemdash and it is easy to show that the flips correspond to the left and right rotations (defined naturally as in a binary search tree) at nodes in a tree. That is, the flip walk on triangulations is isomorphic to the natural random rotation walk on binary trees. For this reason, an important motivating application for the triangulation walk arises in phylogenetic tree reconstruction in bioinformatics~\cite{aldousclad, aldousconj}.

The triangulation walk has so far resisted breakthroughs using the newer algebraic machinery but has seen a handful of results establishing polynomial mixing time bounds. Molloy, Reed, and Steiger~\cite{molloylb} gave the first polynomial upper bound on the mixing time; they used a combinatorial \emph{conductance} argument to show that the mixing time is $O(n^{25})$. They also gave a lower bound of $\Omega(n^{3/2})$ via direct observations about the random walk. McShine and Tetali subsequently improved the upper bound to $O(n^5 \log n)$ in a 1997 paper~\cite{mct} using a framework for comparing Markov chain mixing times, adapting work by Diaconis and Saloff-Coste~\cite{comparison}, as well as Randall and Tetali~\cite{randalltetali}. No new results were established for this problem until Eppstein and Frishberg~\cite{eppstein2022improved} established a bound of $O(n^3 \log^3 n)$ by developing a graph-theoretic decomposition framework\textemdash adapting a prior algebraic decomposition framework by Jerrum, Son, Tetali, and Vigoda~\cite{jstv}.

\subsection{Generalized Associahedra}
In this paper, we study two natural variations on the triangulation flip walk that initially arose in the study of cluster algebras. 

The $n$-dimensional \dfn{associahedron} is a famous polytope whose face lattice is isomorphic to the set of partial triangulations of a regular $(n+3)$-gon ordered by containment of diagonals. Technically speaking, there are several different geometric realizations of the associahedron; they are distinct as polytopes, but they are combinatorially equivalent. We will only concern ourselves with the $1$-skeleton of the $n$-dimensional associahedron, which we denote by $\mathfrak a_{n}$. The vertices of $\mathfrak a_{n}$ correspond to (maximal) triangulations of a regular $(n+3)$-gon, and the edges correspond to triangulation \emph{flips}. This graph is isomorphic to the Hasse diagram of the $(n+1)$-th \emph{Tamari lattice} (in fact, it is isomorphic to the Hasse diagram of any \emph{Cambrian lattice} of type $A_{n}$). It is also the exchange graph of a cluster algebra of type $A_{n}$. 

Roughly 25 years ago, Fomin and Zelevinsky introduced the theory of cluster algebras. Roughly speaking, a \emph{cluster algebra} is an algebra generated by variables that are generated through a dynamical combinatorial procedure known as \emph{mutation}. Some of the most important cluster algebras are those of \emph{finite type}, which are classified by Dynkin diagrams of finite root systems. The \emph{exchange graph} of a finite-type cluster algebra is a finite graph whose edges correspond to mutations. Fomin and Zelevinsky used their theory to generalize associahedra to finite Weyl groups \cite{FominZelevinsky2003}, and Reading later generalized them even further to arbitrary finite Coxeter groups \cite{Reading2006Cambrian}. In this article, we focus on the classical types $A$, $B$, and $D$. The 1-skeleta of the generalized associahedra of types $A_n$, $B_n$ and $D_n$, which we denote by $\mathfrak{a}_n$, $\mathfrak{b}_n$ and $\mathfrak{d}_n$, are the exchange graphs of the cluster algebras of types $A_n$, $B_n$ and $D_n$. As we discuss below, these exchange graphs can also be interpreted combinatorially as flip graphs for certain types of triangulations. 

The famous mixing time problem of sampling triangulations via the flip walk concerns the (classical) associahedra.
The goal of this article is to establish estimates similar to those of Eppstein and Frishberg~\cite{eppstein2022improved} for the mixing times of the simple random walks on the 1-skeleta of generalized associahedra of types $B$ and $D$, which are the only infinite families of classical types other than type~$A$. These random walks can also be interpreted in terms of finite-type cluster algebras; we simply mutate a random vertex of a quiver at each step of the process.

\subsection{Main results}
One can obtain upper bounds on the mixing time of the simple random walk on a graph $G$ from a lower bound on the expansion of $G$. Thus, we prove the following results concerning the expansion and mixing time. 

\begin{theorem}\label{thm:typeB_lower}
The expansion of $\mathfrak b_n$ is  
\[\Omega\left(\frac{1}{\sqrt{n}\log n}\right).\]
The mixing time of the simple random walk on $\mathfrak b_n$ is 
\[O\left(n^3\log^3 n \right).
    \]
\end{theorem}

\begin{theorem}\label{thm:typeB_upper}
The expansion of $\mathfrak b_n$ is  
\[O\left(\frac{1}{\sqrt{n}}\right).\]
The mixing time of the simple random walk on $\mathfrak b_n$ is 
\[\Omega\left(n^{3/2}\right).
    \]
\end{theorem}
\begin{theorem}\label{thm:typeD}
The expansion of $\mathfrak d_n$ is \[\Omega\left(\frac{1}{n^5\log n}\right).\]
The mixing time of the simple random walk on $\mathfrak d_n$ is 
\[O\left(n^{13}\log^2 n \right).
    \]
\end{theorem}

\section{Preliminaries on Expansion and Mixing}

Suppose $(X_t)_{t\geq 0}$ is a discrete-time Markov chain with finite state space $\Omega$. For each choice of an initial state $\xi\in\Omega$, we obtain a sequence $(\mu_t^\xi)_{t\geq 0}$ of probability distributions on $\Omega$ defined by $\mu_t^\xi(S)=\mathbb P(X_t\in S)$. The \dfn{total variation distance} between two probability distributions $\mu$ and $\nu$ on $\Omega$ is the quantity \[d(\mu,\nu)=\frac{1}{2}\sum_{S\subseteq \Omega}|\mu(S)-\nu(S)|.\] If the Markov chain $(X_t)_{t\geq 0}$ is irreducible, then it has a unique stationary distribution $\pi^*$. In this case, for each $\varepsilon>0$, we let $\tau(\varepsilon)$ be the smallest nonnegative integer such that \[\max_{\xi\in\Omega}d(\mu_t^\xi,\pi^*)<\varepsilon\] for all $t\geq \tau(\varepsilon)$; this integer $\tau(\varepsilon)$ is called the \dfn{mixing time} of the Markov chain. 

Let $G=(V,E)$ be a finite simple graph. The \dfn{simple random walk} on $G$ is the Markov chain with state space $V$ defined as follows. If the Markov chain is in state $v$ at time $t$, then it transitions to state $v'$ at time $t+1$, where $v'$ is chosen uniformly at random from the neighbors of $v$.\footnote{For technical reasons, \emph{laziness} is imposed so that the walk stays at $v$ with probability $1/2$.} When $G$ is connected, the simple random walk on $G$ is irreducible, and its stationary distribution $\pi^*$ is such that $\pi^*(v)$ is proportional to the degree of $v$. In particular, if $G$ is a regular graph, then $\pi^*$ is the uniform distribution on $V$. 

For each $S\subseteq V$, let $\partial S=\{\{s,t\}\in E :s\in S,\, t\not\in S\}$. The \dfn{expansion} of $G$ is the quantity 

\begin{equation}
h(G)=\min_{\substack{S\subseteq V \\ |S|\leq|V|/2}}\frac{|\partial S|}{|S|}.
\end{equation}
It is well known~\cite{jsconductance,sinclair_1992} that a lower bound for the expansion of $G$ implies an upper bound on the mixing time of $G$. More precisely, if $G$ is connected and has maximum degree $\Delta$, then the mixing time $\tau(\varepsilon)$ of the simple random walk on $G$ satisfies 
\begin{equation}\label{eq:expmixing}
\tau(\varepsilon)=O\left(\frac{\Delta^2\log(|V|)}{h(G)^2}\right),
\end{equation} 
where the implicit constant depends on $\varepsilon$. 

Let $E^+(G)=\bigcup_{\{u,v\}\in E}\{(u,v),(v,u)\}$ denote the set of all directed edges obtained by replacing each (undirected) edge~$\{u, v\} \in E$ with the directed edges~$(u, v)$ and~$(v, u)$. For $s,t\in V$, a \dfn{flow} from $s$ to $t$ is a function
$\phi\colon E^+ \to \mathbb{R}_{\geq 0}$ such that
\begin{enumerate}[(i)]
\item for all~$u,v \in V$, we have $\min\{\phi(u, v), \phi(v, u)\} = 0$;
\item  $\sum_{v \in V} \phi(s, v) = \sum_{u \in V}\phi(u, t)$;
\item for all $w \in V \setminus \{s, t\}$, we have $\sum_{u \in V} \phi(u, w) = \sum_{v \in V} \phi(w, v)$.
\end{enumerate}

A \dfn{multicommodity flow} in~$G$ is a collection $\Phi=\{\phi_{st}:s,t\in V\}$ of functions
such that $\phi_{st}$ is a flow from $s$ to $t$. In this paper, we will assume that $\Phi$ is \dfn{uniform} in the sense that \[\sum_{v\in V}\phi_{st}(s,v)=\sum_{u\in V}\phi_{st}(u,t)=1\] for all $s,t\in V$. For each $(u,v)\in E^+$, let $\Phi(u,v)=\frac{1}{|V|}\sum_{s,t\in V}\phi_{st}(u,v)$. The \dfn{congestion} of $\Phi$ is the quantity \[\Phi_{\max}=\max_{(u,v)\in E^+}\Phi(u,v).\] 
 
The following result relates flow and expansion. 
\begin{lemma}[\cite{jsconductance,sinclair_1992}]\label{lemma:expansion}
    If $\Phi$ is a multicommodity flow in a graph $G$ with congestion $\Phi_{\max}$, then the expansion of $G$ satisfies \[h(G)\geq\frac{1}{2 \Phi_{\max}}.\] 
\end{lemma}

We aim to use the mixing results concerning the associahedron from \cite{eppstein2021rapid} to show rapid mixing for associahedra of types $B$ and $D$. To accomplish this, we will show how to take a uniform multicommodity flow for a subgraph isomorphic to the associahedron and use it to construct a uniform multicommodity flow for the entire graph. 

Suppose $G = (V, E)$ is a graph and $\mathcal C_1,\ldots,\mathcal C_k$ are subsets of $V$ such that $V=\mathcal C_1\cup\cdots\cup\mathcal C_k$. We define the \dfn{projection graph} $\overline{G}=(\overline V,\overline{E})$ to be the graph with vertex set $\overline V=\{\mathcal C_1,\ldots,\mathcal C_k\}$ in which two distinct vertices $\mathcal C_i$ and $\mathcal C_j$ are adjacent if there exist $u\in \mathcal C_i$ and $v\in\mathcal C_j$ such that $\{u,v\}\in E$ or $u=v$. For $\{\mathcal C_i,\mathcal C_j\}\in\overline{E}$, let $\overline w(\mathcal C_i,\mathcal C_j)$ be the number of pairs $(u,v)\in\mathcal C_i\times\mathcal C_j$ such that $\{u,v\}\in E$ or $u=v$.  
We will usually identify each subset $\mathcal C_i$ with the induced subgraph of $G$ on the vertex set $\mathcal C_i$; in particular, this allows us to speak about multicommodity flows in $\mathcal C_i$. 

The next lemma allows us to use multicommodity flows on subgraphs to construct a multicommodity flow for the entire graph. 

\begin{lemma}
    \label{lem:flowdecomp}
    Let $G = (V, E)$ be a connected graph, and let $\mathcal C_1,\ldots,\mathcal C_k$ be subsets of $V$ whose union is~$V$. Suppose that for each $1\leq i\leq k$, there is a uniform multicommodity flow~$\Phi_i$  in $\mathcal{C}_i$ with congestion at most~$\rho$. Suppose also that there exists a uniform multicommodity flow~$\overline{\Phi}$ in~$\overline{G}$ with congestion ~$\overline\rho$. Then there exists a uniform multicommodity flow~$\Phi$ in~$G$ with congestion at most~$(2\gamma k\overline{\rho} + 1)\rho$, where
\begin{equation}\label{eq:gamma}
\gamma = \max_{i\in[k]}\max_{u\in\mathcal{C}_i}(|\{v\in V \setminus \mathcal{C}_i : (u, v) \in E\}| + |\{j : u \in \mathcal{C}_j\}|).
\end{equation}
\label{lem:projoverlap}
\end{lemma}
Eppstein and Frishberg stated a version of the preceding lemma in terms of Markov chains~\cite[Theorem~13]{eppstein2022improved}. However, they stated this result without the factor of $k$ in the upper bound on the congestion of $\Phi$. They considered a more restricted setting in which $\mathcal C_1,\ldots,\mathcal C_k$ are pairwise disjoint. Allowing the classes to intersect nontrivially could reduce the congestion~$\overline\rho$ of the flow in the projection graph. It is straightforward to show (by tracing through the proof of \cite[Theorem~13]{eppstein2022improved}) that we pay for this improvement with an additional factor of at most $k$. This is because each flow within a class~$\mathcal C_i$ can now also produce congestion in more than one of the~$k$ classes. 

Following the reasoning in~\cite[Lemma 16]{eppstein2022improved} while accounting for potential nontrivial intersections of the classes (as in \cref{lem:projoverlap}) yields the following. Epppstein and Frishberg proved and used this (dealing with nontrivial intersections) in their paper but did not state it explicitly:
\begin{corollary}\label{corollary}
   Let $G = (V, E)$ be a graph, and let $\mathcal C_1,\ldots,\mathcal C_k$ be subsets of $V$ whose union is $V$. Suppose that for each $1\leq i\leq k$, there is a uniform multicommodity flow~$\Phi_i$ in $\mathcal{C}_i$ with congestion at most~$\rho$. 
   Then there exists a uniform multicommodity flow~$\Phi$ in~$G$ with congestion at most
    \[\left(\frac{2\gamma |V|}{\mathcal{E}_{\min}}+1\right)k\rho,\] where $\mathcal E_{\min}=\min\{\overline w(\mathcal C_i,\mathcal C_j):\overline w(\mathcal C_i,\mathcal C_j)>0\}$ and $\gamma$ is as in \eqref{eq:gamma}.
\end{corollary}
Since any given multicommodity flow $\phi$ in a graph~$G$ is still valid if edges are added to~$G$, we observe that \cref{corollary} still holds with the same bound if for some pairs~$\mathcal{C}_i, \mathcal{C}_j$, the set of edges between~$\mathcal{C}_i$ and~$\mathcal{C}_j$ is removed from the projection graph so long as the resulting graph is still connected. 

We use the well known fact that the number of vertices of the associahedron $\mathfrak a_n$ is the $(n+1)$-th \dfn{Catalan number} 
$$C_{n+1} = \frac{1}{n+2}\binom{2n+2}{n+1}.$$
The number of vertices of $\mathfrak b_n$ is $\binom{2n}{n}$, and the number of vertices of $\mathfrak d_n$ is $\frac{3n-2}{n}\binom{2n-2}{n-1}$. 

The main result of Eppstein and Frishberg (written in terms of congestion) is as follows.
\begin{lemma}[\cite{eppstein2022improved}]
\label{lem:eppfrishtri}
There exists a uniform multicommodity flow in $\mathfrak a_n$ with congestion~$O(\sqrt n \log n)$.
\end{lemma}

\section{Associahedra of Type $B$}
\label{sec:typeb}
We write $\mathfrak b_n$ for the $1$-skeleton of the $n$-dimensional associahedron of type $B$ (also called the \emph{cyclohedron}). We begin this section with a combinatorial model for $\mathfrak b_n$. The vertex set consists of centrally-symmetric triangulations of a regular $(2n+2)$-gon $P_{2n+2}$; two such triangulations are connected by an edge if one can be obtained from the other via one of the following \dfn{flips}:
\begin{enumerate}
\item a diagonal flip that replaces one diagonal of the $(2n+2)$-gon with another diagonal; 
\item a pair of centrally symmetric diagonal flips, which replaces two centrally symmetric diagonals with two other centrally symmetric diagonals. More explicitly, consider the unique quadrilateral that contains a given diagonal. Then that diagonal is flipped to the other diagonal of this quadrilateral.  
\end{enumerate}

The rules for flipping diagonals in $\mathfrak{b}_n$ are similar to those of $\mathfrak{a}_n$ except the flips maintain the symmetry of the polygon. \cref{fig:typeB} shows $\mathfrak b_3$. 

\begin{figure}
    \centering
    \includegraphics[width = 0.6\textwidth]{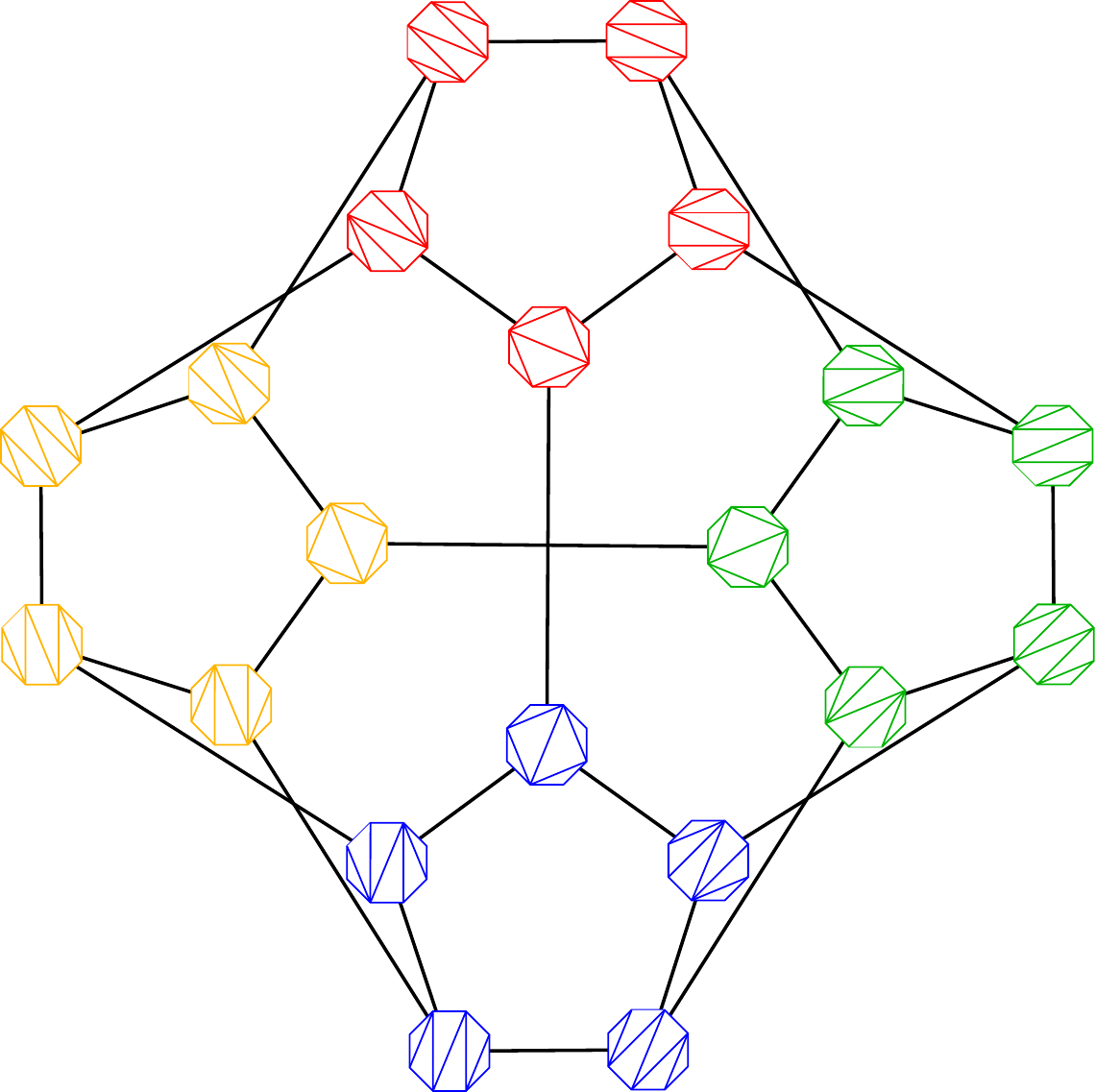}
    \caption{The graph $\mathfrak b_3$. Triangulations of the same color belong to the same class, meaning they have the same central diagonal.}
    \label{fig:typeB}
\end{figure}

A diagonal of $P_{2n+2}$ is \dfn{central} if it passes through the center of $P_{2n+2}$. Each centrally-symmetric triangulation of $P_{2n+2}$ contains a unique central diagonal. Thus, it makes sense to partition the vertices of $\mathfrak b_n$ by their central diagonals: given a central diagonal $D$, let $\Psi(D)$ denote the set of centrally-symmetric triangulations of $P_{2n+2}$ that use $D$. We will identify $\Psi(D)$ with the subgraph of the $1$-skeleton of $\mathfrak b_n$ that it induces. 
\begin{remark}\label{remark:typeB}
If $D$ is a central diagonal of $P_{2n+2}$, then each triangulation $T\in \Psi(D)$ is uniquely determined by the diagonals that $T$ uses on one particular side of $D$ (because $T$ is centrally symmetric). These diagonals create a triangulation of an $(n+2)$-gon. Thus, $\Psi(D)$ is isomorphic to $\mathfrak a_{n-1}$. For example, each of the four classes in \cref{fig:typeB} is a $5$-cycle, which is isomorphic to $\mathfrak a_2$. 
\end{remark}

Given diagonals $D$ and $D'$ of $P_{2n+2}$, let  $\mathcal{E}(D, D')$ be the set of edges of $\mathfrak b_n$ with one endpoint in $\Psi(D)$ and one endpoint in $\Psi(D')$.

\begin{figure}
    \centering
    \includegraphics[height=4cm]{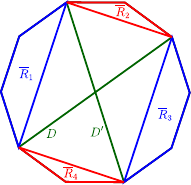}
    \caption{An illustration of part of the proof of \Cref{thm:typeBloose} with $n=4$. Selecting an edge in $\mathcal E(D,D')$ is equivalent to choosing triangulations of $\overline R_1$ and $\overline R_2$. }
    \label{fig:typeB_decomp}
\end{figure}

In this section, we prove a lower bound on the expansion of the type-$B$ associahedron that is weaker than the one stated in \cref{thm:typeB_lower} in order to illustrate our key ideas. The proof of the tighter bound is more involved and is deferred to the appendix. 

\begin{proposition}\label{thm:typeBloose}
The expansion of $\mathfrak b_n$ is  
\[\Omega\left(\frac{1}{n^{4}\log n}\right).\]
The mixing time of the simple random walk on $\mathfrak b_n$ is 
\[O\left(n^{11}\log^2 n \right).
    \]
\end{proposition}

\begin{proof}
Our strategy is to apply \cref{corollary} to the setting in which $G$ is $\mathfrak b_n$. To this end, let $D_1,\ldots,D_{n+1}$ be the central diagonals of $P_{2n+2}$, and let $\mathcal C_i=\Psi(D_i)$. (We are setting $k=n+1$ in \cref{corollary}.)

Let $D$ and $D'$ be distinct central diagonals of $P_{2n+1}$. Let $Q$ be the quadrilateral whose vertices are the endpoints of $D$ and $D'$. Let $R_1,R_2,R_3,R_4$ be the sides of $Q$, listed clockwise. For $1\leq r\leq 4$, let $\overline R_r$ be the polygon formed by $R_r$ and the sides of $P_{2n+2}$ that are not on the same side of $R_r$ as $Q$. Each edge in $\mathcal E(D,D')$ must flip $D$ into $D'$. Choosing such an edge is equivalent to choosing a triangulation of $\overline R_1$ and a triangulation of $\overline R_2$ (which then determine triangulations of $\overline R_3$ and $\overline R_4$ by central symmetry). (See \cref{fig:typeB_decomp}.)
Let $k_1$ (respectively, $k_2$) be the number of sides of $\overline R_1$ (respectively, $\overline R_2$) that are also sides of $P_{2n+2}$. Then $k_1+k_2=n+1$. The number of triangulations of $\overline R_1$ (respectively, $\overline R_2$) is the Catalan number $C_{k_1-1}$ (respectively, $C_{k_2-1}$). Thus, $|\mathcal E(D,D')|=C_{k_1-1}C_{k_2-1}$. 

It follows from the preceding paragraph that 
\begin{equation}
\mathcal E_{\min}=\min_{\substack{k_1,k_2\geq 1 \\ k_1+k_2=n+1}}C_{k_1-1}C_{k_2-1}=C_{\left\lfloor (n+1)/2\right\rfloor-1}C_{\left\lceil (n+1)/2\right\rceil-1}=\Omega(n^{-3}4^n).
\label{eq:emincc}
\end{equation}
The number of vertices of $\mathfrak b_n$ is $|V| = \binom{2n}{n}=O(n^{-1/2}4^n)$, so 
 \begin{equation}
  \frac{|
  V|}{\mathcal{E}_{\min}} =  \Omega(n^{5/2}).
 \end{equation}

For each class $\Psi(D)$ and each vertex $u\in\Psi(D)$, there is at most one edge connecting $u$ to a vertex that is not in $\Psi(D)$ corresponding to the flip of $D$. Since each vertex of $\mathfrak b_n$ belongs to a unique class (i.e., each centrally symmetric triangulation has exactly one central diagonal), the quantity $\gamma$ defined in \eqref{eq:gamma} is $2$. Each class $\Psi(D)$ is isomorphic to the $1$-skeleton of $\mathfrak a_{n-1}$, so we can combine \cref{lem:eppfrishtri} with \cref{corollary} to conclude that there exists a uniform multicommodity flow in the $1$-skeleton of $\mathfrak b_n$ with congestion at most 
\begin{equation}\label{eq:typeB}
\left(\frac{2\gamma |V|}{\mathcal{E}_{\min}} + 1\right) k\cdot O(\sqrt{n}\log n) = O(n^{4}\log n).
\end{equation}    
According to \cref{lemma:expansion}, the expansion of $\mathfrak b_n$ satisfies \[h(\mathfrak b_n)=\Omega\left(\frac{1}{n^{4}\log n}\right).\]  
Since $\mathfrak b_n$ is a regular graph of degree $n$, we conclude from \eqref{eq:expmixing} that the mixing time $\tau(\varepsilon)$ of the simple random walk on $\mathfrak b_n$ satisfies 
\begin{equation}     \tau(\varepsilon)=O\left(\frac{(n-1)^2 \log(\binom{2n-2}{n-1})}{h(\mathfrak b_n)^2} \right) =     O\left(n^{11}\log^2 n \right). \qedhere
\end{equation}
\end{proof}

\section{An Upper Bound on the Expansion of $\mathfrak b_n$}
\label{sec:upper}
To obtain the upper bound on expansion stated in \cref{thm:typeB_upper}, we just need to show that there exists a cut $(S, V\backslash S)$ such that $\frac{|\partial S|}{|S|} = O(\frac{1}{\sqrt{n}})$, where $\partial S = \{(u, v) : u \in S,\, v \in V \backslash S\}$.

For this construction, number the vertices of the regular $(2n+2)$-gon $P_{2n+2}$ as $1,\dots,2n+2$. Let $\Xi$ be the set of central diagonals of $P_{2n+2}$ that have a vertex in $[1,(n+1)/2]$. Let \[S = \bigcup_{D \in\Xi}\Psi(D)\quad\text{and}\quad\overline S=\bigcup_{D\not\in\Xi}\Psi(D)=V\setminus S.\] 
We will show that the number of edges between $S$ and $\overline{S}$ is small. 

    \begin{figure}
    \centering
    \includegraphics[height=5cm]{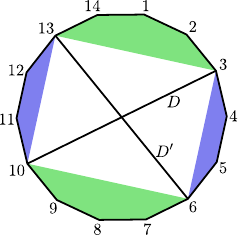}
    \caption{Two central diagonals $D$ and $D'$ of $P_{14}$ with $\Psi(D)\subseteq S$ and $\Psi(D')\subseteq \overline S$. The rectangle whose diagonals are $D$ and $D'$ breaks $P_{14}$ into two $4$-gons (blue) and two $5$-gons (green), so $d(D,D')=3$.}
    \label{fig:S decomposition}
\end{figure}

Consider central diagonals $D,D'$ such that $\Psi(D) \subseteq S$ and $\Psi(D') \subseteq \overline{S}$. Every triangulation in $\mathcal E(D,D')$ includes as diagonals the four sides of the rectangle whose diagonals are $D$ and $D'$. There is a unique integer $d=d(D,D')$ satisfying $1\leq d\leq (n+1)/2$ such that the parts of $P_{2n+2}$ outside of this rectangle are two $(d+1)$-gons and and two $(n-d+2)$-gons. See \cref{fig:S decomposition}. This decomposition implies that \[|\mathcal E(D,D')|=C_{d-1}C_{n-d}.\]

\begin{lemma}\label{lem:EDD'} We have 
 \[
    \frac{|\mathcal{E}(D, D')|}{|\Psi(D)|}= \frac{|\mathcal{E}(D, D')|}{|\Psi(D')|} = O\left(\frac{1}{d(D, D')^{3/2}}\right).
\]
\end{lemma}
\begin{proof}
As mentioned above, we have $|\mathcal{E}(D, D')|~=~C_{d-1} C_{n-d}$. Therefore,
 \[
        \frac{|\mathcal{E}(D, D')|}{|\Psi(D)|} = O\left(  \frac{\frac{1}{(d-1)^{3/2}} \cdot\frac{1}{(n-d)^{3/2}}4^{d-1+n-d}}{\frac{1}{n^{3/2}} 4^n}\right) = O\left(\frac{1}{d^{3/2}}\right),
   \]
as desired.
\end{proof}

\cref{lem:EDD'} tells us that if $d(D,D')$ is large, then there are few edges between the classes $\Psi(D)$ and $\Psi(D')$.

\begin{proof}[Proof of \cref{thm:typeB_upper}] 
Let $y(D)=\min\limits_{D'\not\in\Xi}d(D, D')$. By replacing $d$ with the continuous integration variable $x$, we find that
\begin{align*}
    |\partial S| &= \sum_{D \in \Xi} \sum_{D' \not\in\Xi} |\mathcal{E}(D, D')|\\
    &=O\left( 2\sum_{D \in S} \int_{y(D)}^{n/2 }\frac{1}{x^{3/2}} |\Psi(D)| dx\right)\\
    &= O\left( \int_{1}^{n/2}\int_{y}^{n/2} \frac{1}{x^{3/2}}C_ndx dy\right)\\
    &=O\left(\frac{|V|}{n}\int_{1}^{n/2}\frac{1}{\sqrt{y}} dy\right)\\
    &= O\left(\frac{|V|}{n}\sqrt{n}\right)\\
    &= O\left(\frac{|V|}{\sqrt{n}}\right),
\end{align*}
where in the second equality, we used the fact that $y(D) \in [1, (n+1)/4]$ and the fact that for each $k\neq (n+1)/4$, there are $2$ different central diagonals $D$ such that $y(D) = k$.
Since $|S| = \Theta\left(\frac{|V|}{2}\right)$ it follows that 
\[    h(\mathfrak{b}_n) \leq \frac{|\partial S|}{|S|} = O\left(\frac{1}{\sqrt{n}}\right),
\]
as desired.
    \end{proof}

\section{Associahedra of Type $D$}

Let $\mathfrak d_n$ denote the $n$-dimensional associahedron of type $D$. We use the following combinatorial model of $\mathfrak d_n$. 

Begin with a regular $2n$-gon $P_{2n}$. Choose a small positive number $\varepsilon$, and let $\bullet$ be the disc of radius $\varepsilon$ whose center coincides with that of $P_{2n}$. Let us assume that $\varepsilon$ is chosen small enough so that the only diagonals of $P_{2n}$ intersecting $\bullet$ are the central diagonals. Let $P_{2n}^\bullet$ be the punctured $2n$-gon that results from removing $\bullet$ from $P_{2n}$. 

Define a \dfn{central chord} of $P_{2n}^\bullet$ to be a line segment $L$ such that one endpoint of $L$ is a vertex of $P_{2n}$ while the other endpoint of $L$ is a point at which $L$ is tangent to $\bullet$. A \dfn{chord} of $P_{2n}^\bullet$ is a line segment that is either a non-central diagonal of $P_{2n}$ or a central chord of $P_{2n}^\bullet$. A \dfn{triangulation} of $P_{2n}^\bullet$ is a maximal collection of chords of $P_{2n}^\bullet$ that do not cross each other. Say such a triangulation is \dfn{centrally symmetric} if it is invariant under $180^\circ$-rotation about the center of $P_{2n}^\bullet$. The vertices of $\mathfrak d_n$ correspond to centrally symmetric triangulations of $P_{2n}^\bullet$. We can perform a \dfn{flip} on a centrally symmetric triangulation of $P_{2n}^\bullet$ by deleting a pair of centrally symmetric chords and replacing them with the unique other pair of centrally symmetric chords that can be added to form a triangulation. Two vertices of $\mathfrak d_n$ are adjacent if one can be obtained from the other via a flip. \cref{fig:typeD} shows $\mathfrak d_3$. 

\begin{figure}
    \centering
    \includegraphics[width = .5\textwidth]{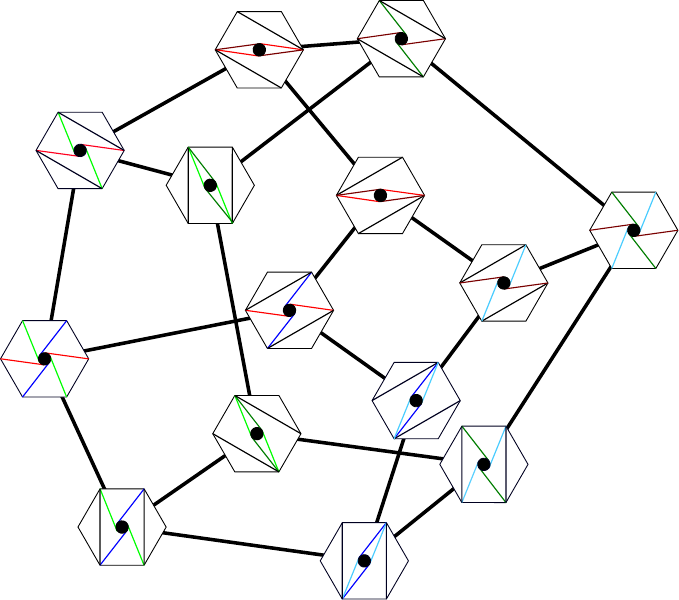}
    \caption{The graph $\mathfrak d_3$. We have color-coded the centrally symmetric pairs of central chords. }
    \label{fig:typeD}
\end{figure}

The central chords of $P_{2n}^\bullet$ come in centrally symmetric pairs; for each such pair $\dd$ of $P_{2n}^\bullet$, let $\Psi(\dd)$ be the set of triangulations of $P_{2n}^\bullet$ that use $\dd$. Every triangulation of $P_{2n}^\bullet$ contains at least one pair of centrally symmetric central chords, so the union of these classes of the form $\Psi(\dd)$ is the entire vertex set of $\mathfrak d_n$. Note that there exist distinct pairs $\dd$ and $\dd'$ such that $\Psi(\dd)\cap\Psi(\dd')$ is nonempty.

Suppose $\rr$ is a pair of opposite vertices of $P_{2n}^\bullet$. Let $\clock\rr$ (respectively, $\counter\rr$) be the pair of centrally symmetric central chords that use the vertices in $\rr$ and emanate from these vertices in the clockwise (respectively, counterclockwise) direction. (See \cref{fig:clockcounter}.) 

\begin{figure}
    \centering
    \includegraphics[height=1.88cm]{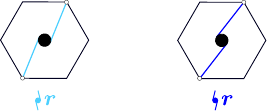}
    \caption{The pairs $\clock\rr$ and $\counter\rr$ of central chords in $P_{6}^\bullet$, where $\rr$ is the pair of vertices marked with small circles.}  
    \label{fig:clockcounter}
\end{figure}

\begin{lemma}\label{lem:TypeDClasses}
Let $\rr$ be a pair of opposite vertices of $P_{2n}^\bullet$. Then $\Psi(\clock\rr)$ and $\Psi(\counter\rr)$ are both isomorphic to $\mathfrak a_{n-1}$, and $\Psi(\clock\rr)\cap\Psi(\counter\rr)$ is isomorphic to $\mathfrak a_{n-2}$. 
\end{lemma}

\begin{proof}
Note first that $\Psi(\clock\rr)$ and $\Psi(\counter\rr)$ are isomorphic to each other. We obtain an $(n+2)$-gon from $P_{2n}^\bullet$ by shrinking the disc $\bullet$ into a vertex and then deleting half of the polygon $P_{2n}$. The isomorphism between $\Psi(\clock\rr)$ and $\mathfrak a_{n-1}$ is illustrated in \cref{fig:isomorphism} for $n=6$. A triangulation in $\Psi(\clock\rr)$ uses $\counter\rr$ (i.e., is in $\Psi(\clock\rr)\cap\Psi(\counter\rr)$) if and only if the corresponding triangulation of the $(n+2)$-gon uses the diagonal connecting the two vertices in $\rr$. This yields the desired isomorphism between $\Psi(\clock\rr)\cap\Psi(\counter\rr)$ and $\mathfrak a_{n-2}$. 
\end{proof}

\begin{figure}
    \centering
    \includegraphics[height=8.09cm]{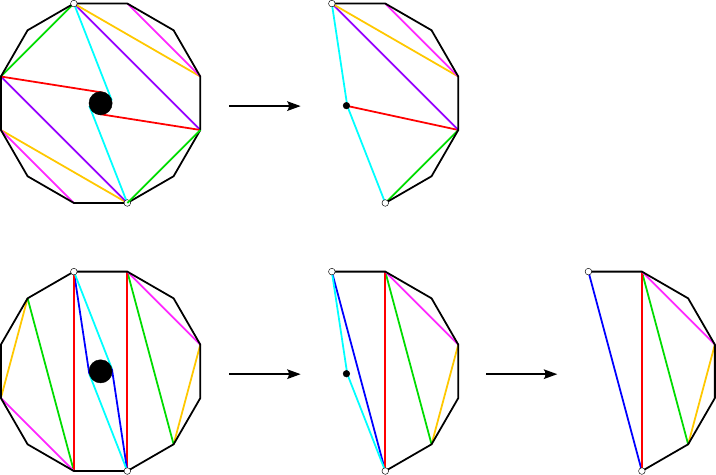}
    \caption{An illustration of \cref{lem:TypeDClasses} for $n=6$. The top and bottom images are two examples of the isomorphism between $\Psi(\clock\rr)$ and $\mathfrak a_{n-1}=\mathfrak a_5$. The image on the bottom illustrates the restriction of this isomorphism to $\Psi(\clock\rr)\cap\Psi(\counter\rr)$, which is an isomorphism from $\Psi(\clock\rr)\cap\Psi(\counter\rr)$ to $\mathfrak a_{n-2}=\mathfrak a_4$. }  
    \label{fig:isomorphism}
\end{figure}

\begin{proof}[Proof of \cref{thm:typeD}]

Let $\dd_1,\ldots,\dd_{2n}$ be the centrally symmetric pairs of central chords of $P_{2n}^\bullet$, and let $\mathcal C_i=\Psi(\dd_i)$. We apply \Cref{corollary} with the classes $\mathcal C_1,\ldots,\mathcal C_{2n}$. (We are setting $k=2n$ in \Cref{corollary}.) \cref{lem:TypeDClasses} tells us that each class $\mathcal C_i$ is isomorphic to $\mathfrak a_{n-1}$, so we know that it contains a multicommodity flow with congestion $O(\sqrt n \log n)$ by \cref{lem:eppfrishtri}. 

We apply \cref{corollary} with $\mathcal{C}_i$ equal to some $\Psi(\clock \rr)$ for $1 \leq i \leq n$ and to some $\Psi(\counter \rr)$ for $ n + 1 \leq i \leq 2n$. To compute~$\mathcal{E}_{\min}$ (used in \cref{corollary}), we need to compute the sizes of the intersections of each pair of classes. We have
\[\mathcal{E}_{\min} = \min\left\{\min_{\rr} |\Psi(\clock \rr) \cap \Psi(\counter \rr)|, \min_{\boldsymbol{q},\rr} |\Psi(\clock \boldsymbol{q}) \cap \Psi(\clock \rr)|, \min_{\boldsymbol{q}, \rr} |\Psi(\counter \boldsymbol{q}) \cap \Psi(\counter \rr)|\right\}\]
since the only nontrivial intersections between classes occur between pairs of the form $\Psi(\clock \rr), \Psi(\counter \rr)$, $\Psi(\clock \boldsymbol{q}), \Psi(\clock \rr)$, or $\Psi(\counter \boldsymbol{q}), \Psi(\counter \rr)$. We know by \cref{lem:TypeDClasses} that 
$ |\Psi(\clock \rr) \cap \Psi(\counter \rr)| = C_{n-1}$. Moreover, for distinct pairs $\boldsymbol{q},\rr$ of oppositie vertices of $P_{2n}^\bullet$, it follows from the isomorphism used to prove \cref{lem:TypeDClasses} that 
\[
    |\Psi(\clock q) \cap \Psi(\clock r)|= C_{i}C_{j}
\]
for some $i,j$ with $i + j = n$. By symmetry, we have $|\Psi(\clock r) \cap \Psi(\clock q)| = |\Psi(\counter r) \cap \Psi(\counter q)|$. By the same reasoning used to derive \eqref{eq:emincc}, we have
\[|\Psi(\clock r) \cap \Psi(\clock q)| = |\Psi(\counter r) \cap \Psi(\counter q)|\geq C_{\lfloor n/2\rfloor - 1}C_{\lceil n/2\rceil - 1} = \Omega(n^{-3}4^n).\] Consequently,
\[
\mathcal{E}_{\min} \geq \min\{C_{n-2}, \Omega(n^{-3}4^n)\} = \Omega(n^{-3}4^n)
.\]
Since $|V(\mathfrak{d}_n)| = \frac{3n-2}{n}\binom{2n-2}{n-1} = \Theta(\frac{1}{\sqrt n}4^n)$, this implies that
\begin{equation}
    \mathcal{E}_{\min} \geq \Omega(n^{-5/2})|V(\mathfrak{d}_n)|.
    \label{eq:emind}
\end{equation}
Applying \cref{corollary}, we find that a multicommodity flow exists in~$\mathfrak{d}_n$ with congestion at most
\begin{equation}
\left(\frac{2\gamma |V|}{\mathcal{E}_{\min}}+1\right)k\rho \leq
O(n^{5/2} \cdot n^2\cdot \sqrt n \log n) = O(n^5 \log n),
\end{equation}
where the term $O(n^{5/2})$ comes from \eqref{eq:emind}, the term $O(n^2)$ comes from the observation that~$\gamma$ and~$k$ are both $O(n)$, and the term~$\rho = O(\sqrt n \log n)$ comes from applying \cref{lem:eppfrishtri} to each class. 
\end{proof}

\section{Discussion of Methods and Prior Work}
We have used the multicommodity flow construction of~\cite{eppstein2022improved} to obtain new rapid (polynomial) mixing results for the natural random walk on the type-$B$ and type-$D$ associahedra. Our expansion bound for the type-$B$ associahedron is nearly tight (up to logarithmic factors). We observe, however, that for the type-$B$ associahedron, we could have combined our decomposition technique with prior work that predates Eppstein and Frishberg~\cite{eppstein2022improved}. That is, a coarse polynomial bound for the type-$B$ associahedron could be obtained from prior results establishing $O(n^5 \log n)$ mixing for the type-$A$ associahedron. To apply the $O(n^5 \log n)$ bound for type $A$ to the type-$B$ associahedron, one can use a pre-existing analogue of \cref{lem:flowdecomp} that uses the \emph{spectral gap} of the chain. The theorem, due to Jerrum, Son, Tetali, and Vigoda~\cite{jstv}, with some additional work, gives a mixing time bound of $O(n^{6.5}\log^2 n)$ when combined with our decomposition. When applied to Eppstein and Frishberg's $O(n^3 \log^3 n)$ result, the theorem gives $O(n^{4.5} \log^3 n)$. However, we needed the flow machinery in in~\cite{eppstein2022improved} to obtain the tight expansion bound and corresponding $O(n^3 \log^3 n)$ mixing bound.

For the type-$D$ associahedron, we are not aware of a prior decomposition framework that can be applied in a black-box fashion to obtain polynomial mixing, without the modifications due to \cite{eppstein2021rapid} that we have stated in \cref{corollary}.

\bibliographystyle{plainurl}
\bibliography{references}

\appendix
\newpage
\section{A Nearly Tight Lower Expansion Bound for Type $B$}
\subsection{Setup and MSFs}\label{sec:setup}
In this section, we will do a more careful analysis in order to derive an expansion lower bound that matches the upper bound up to logarithmic factors. More specifically, we prove \cref{thm:typeB_lower} by constructing a (uniform) multicommodity flow in $\mathfrak{b}_n = (V, E)$ (the type-$B$ associahedron over the $2n$-gon): one in which every triangulation in $V$ sends a unit of a commodity to every other triangulation in $V$. Consider the partitioning into classes $V = \bigcup_D \Psi(D)$ in \cref{sec:upper}. All of the classes are isomorphic to each other and to a type-$A$ associahedron $\mathfrak{a}_n$.

\cref{lem:eppfrishtri} implies that any two  triangulations~$s,t$ in a given class $\Psi(D)$ can exchange a unit of flow while generating congestion at most~$O(\sqrt n \log n)$\textemdash actually, when considering normalization factors in the definition of congestion, it is at most~$\frac{|\Psi(D)|}{|V|}\cdot O(\sqrt{n}\log n)=O(\log n/\sqrt n)$ .

The harder part is to find the paths through which a triangulation $s\in \Psi(D)$ sends flow to a triangulation $t \in \Psi(D')$. We follow the main steps of~\cite{eppstein2022improved} in constructing this flow:
\begin{enumerate}[(i)]
\item ``Shuffle'' the total flow being sent by $s\in\Psi(D)$ to other triangulations, by distributing an equal amount of this outbound flow to each triangulation~$s' \in \Psi(D)$.
\item ``Concentrate'' this outbound flow within the \emph{boundary set} $\partial_{D'}(D)$\textemdash the set of triangulations in~$\Psi(D)$ having a neighbor in~$\Psi(D')$.
\item ``Transmit'' the flow that is bound for triangulations in~$\Psi(D')$ across the matching $\mathcal{E}(D, D')$.
\item ``Distribute'' the incoming flow from the matching $\mathcal{E}(D, D')$ (flow that originated at~$s\in\Psi(D)$) uniformly throughout~$\Psi(D')$.
\end{enumerate}
Formally, this flow is defined as a composition of \emph{multi-way single-commodity flows} (\emph{MSFs}). An \emph{MSF problem} (defined in~\cite{eppstein2022improved}) is a tuple
$$(\sigma, \delta, S\subseteq V, F\subseteq V),$$
where~$S$ is a \emph{source set} of vertices each of which has $\sigma$ units of flow that it must send out, and $F$ is a \emph{sink set} of vertices each of which has~$\delta$ units of flow that it must receive. There is only a single commodity to be transmitted, unlike in a multicommodity flow.  
When~$\delta$ is undefined, we assume~$\delta = |S|\sigma/|F|.$ We may also define~$\sigma: S \rightarrow \mathbb{R}_{\geq 0}$ and $\delta: F \rightarrow \mathbb{R}_{\geq 0}$ as functions, in which case each $v \in S$ (respectively, each $u \in F$) must send out $\sigma(v)$ (respectively, must receive $\delta(u)$) units of flow.

An \dfn{MSF} is a flow function~$\phi$ that ``solves'' an MSF problem; that is, the net flow out of each vertex in~$S \setminus F$ is~$\sigma$, the net flow into each vertex in~$F \setminus S$ is~$\delta$, the net flow out of each vertex in~$S \cap F$ is~$\sigma - \delta$, and the net flow out of each vertex in~$V \setminus (S \cup F)$ is zero.

\subsection{Solving MSFs}
\label{sec:steps}
We can view the shuffling step of sending flow from a given~$s$ as an MSF problem
$$\left(\sigma = |V|, \delta = \frac{|V|}{|\Psi(D)|}, S = \{s\}, F = \Psi(D)\right).$$
Thus, the overall shuffling step consists of~$|\Psi(D)|$ separate MSFs, constructed via the uniform multicommodity flow described above by \cref{lem:eppfrishtri}\textemdash since each class is isomorphic to a type-$A$ associahedron. The advantage of the shuffling step is that the remaining steps can be handled in aggregate, with a single commodity (MSF), for all~$s \in \Psi(D)$\textemdash since the sink set~$F$ and demand~$\delta$ of the shuffling step are identical for all flows from source vertices~$s$. 

The ``transmission'' step is trivial to define, and we do so in aggregate: let each edge in~$\mathcal{E}(D, D')$ carry an equal amount of flow, namely~$\frac{|\Psi(D)||\Psi(D')|}{|\mathcal{E}(D, D')|}$.
In other words, we have
$$\left(\sigma = \delta = \frac{|\Psi(D)||\Psi(D')|}{|\mathcal{E}(D, D')|}, S = \partial_{D'}(D), F = \partial_D(D')\right),$$
where $\partial_{D'}(D)$ denotes the set of vertices in $\Psi(D)$ that have at least one neighbor in $\Psi(D')$. 

The ``distribution'' step and the ``concentration'' step are the same up to symmetry, so it suffices to define the distribution step and bound its congestion: consider a class $\Psi(D)$, which is isomorphic to a type-$A$ associahedron~$\mathfrak{a}_n$, and which has received $\frac{|\Psi(D)||\Psi(D')|}{|\mathcal{E}(D, D')|}$ across each edge in~$\mathcal{E}(D, D')$ from each~$D'$.

We have, in the distribution step, an MSF problem
$$\left(\sigma = \frac{|\Psi(D)||\Psi(D')|}{|\mathcal{E}(D, D')|}, \delta = |\Psi(D')|, S = \partial_{D'}(D), F = \Psi(D)\right)$$
for each~$D'$; that is, the source set is~$\partial_D(D')$, and each triangulation in this source set is ``carrying''~$\sigma = \frac{|\Psi(D)||\Psi(D')|}{|\mathcal{E}(D, D')|}$ flow (which originated at vertices in~$\Psi(D')$).

\subsection{Solving the distribution MSFs: $O(n)$ overall congestion}
\label{sec:oriented}
We analyze the boundary set~$\partial_D(D')$: Eppstein and Frishberg~\cite{eppstein2022improved} observed that one can decompose the state space of~$\mathfrak{a}_n$ in a different way: orient the $n$-gon with an arbitrarily labeled ``special'' polygon edge~$e^*$ on the bottom. Every triangulation includes~$e^*$ since~$e^*$ is an edge of the polygon. Furthermore, every triangulation includes exactly one triangle~$T$ that has~$e^*$ as a side: partition the triangulations in~$\mathfrak{a}_n$ into a collection of classes~$\hat\Psi(T)$, where each~$T$ is a triangle having~$e^*$ as an edge, and a triangulation~$t$ belongs to class~$\hat\Psi(T)$ if~$t$ includes the triangle~$T$.

\begin{lemma}[\cite{eppstein2022improvedfull}, Lemma 38]
    \label{lem:logtrick}
    Consider the type-$A$ associahedron~$\mathfrak{a}_n$ oriented with a special bottom edge~$e^*$. Consider the triangles~$T_1, T_2, \dots, T_n$ that include~$e^*$ as an edge (labeled in clockwise order according to their topmost vertex).
    
    Consider~$n$ MSF problems, each of the form
    $$(\sigma_i, S = \hat\Psi(T_i), F = V(\mathfrak{a}_n)),$$
    and suppose that $\sigma_i \leq |V(\mathfrak{a}_n)|$ for all $i$. Then this collection of MSF problems can be solved with total congestion~$O(\sqrt n)$ within~$\mathfrak{a}_n$.
\end{lemma}

Now fix~$D$, and consider any diagonal~$D_i$, where~$i$ is a vertex on the~$2n$-gon and~$D_i$ is the central diagonal of the~$2n$-gon having~$i$ as an endpoint. The boundary set~$\partial_{D_i}(D)$ can be identified with a class~$\hat\Psi(T_i)$ in the type-$A$ associahedron~$\Psi(D) \simeq \mathfrak{a}_n$, where the triangle~$T_i$ is found by setting~$e^*$ to the diagonal~$D$, setting the third vertex of triangle~$T_i$ to the vertex~$i$. Furthermore, we have
$$\sigma_i = \frac{|\Psi(D)||\Psi(D_i)|}{|\mathcal{E}(D, D_i)|} \leq O(\sqrt n)|V(\mathfrak{b}_n)|$$
by \cref{lem:EDD'} and by the fact that~$|\Psi(D)| = |\Psi(D_i)| = |V(\mathfrak{b}_n)|/n$. 

Combining these observations with \cref{lem:logtrick}, we conclude that the concentration and distribution MSF problems described in this section can be solved while producing~$O(n)$ total congestion in a given class~$\Psi(D)$.
Adding this congestion to the~$O(\sqrt n)$ congestion produced by the transmission step and the $O(\sqrt n \log n)$ congestion in the shuffling step gives a flow with overall congestion~$O(n)$ in $\mathfrak{b}_n$, leading to expansion~$\Omega(1/n)$.

\subsection{Tightening the congestion bound}
In this section, we obtain the nearly tight congestion bound in \cref{thm:typeB_lower} by following another technique from~\cite[Lemma 41]{eppstein2022improvedfull}. This improves the concentration and transmission steps in \cref{sec:setup}.

We will need the following additional lemma. 
\begin{lemma}
\label{lem:singlemsf}
Consider an MSF problem in $\mathfrak{a}_n$ in which the source set~$S$ is a union of classes $\bigcup_i \hat\Psi(T_i)$, the sink set~$F$ is all of~$V(\mathfrak{a}_n)$, and the surplus function~$\sigma$ is everywhere at most~$|V(\mathfrak{a}_n)|$ and is uniform over any given $\hat\Psi(T_i)$. Any such MSF can be solved with congestion at most $1$ in~$\mathfrak{a}_n$. 
\end{lemma}
\begin{proof}
    This follows from tracing the proof of \cite[Lemma 23]{eppstein2022improvedfull} with a single commodity (MSF) instead of a multicommodity flow. The extra factor of~$\kappa = \Theta(n)$ otherwise incurred in that proof is seen to vanish when considering an MSF using the observations about the concentration flow at the end of the proof. An earlier version of their paper (\cite[Lemma A.10]{eppstein2022improvedv1}) contained the precise decomposition. For completeness, we summarize this construction. 

    The MSF problem is of the form
    $$(\sigma, \delta, S = \bigcup_i \hat\Psi(T_i), F = V(\mathfrak{a}_n)),$$
    where $\sigma(t) = \sigma(t')$ whenever $t,t'$ are in the same class~$\hat\Psi(T_i)$, but where $\sigma$ may differ across classes $\hat\Psi(T_i), \hat\Psi(T_j)$\textemdash and where~$\delta$ is uniform. Write $\sigma(T_i)$ to denote the value of $\sigma$ at each $t \in \hat\Psi(T_i)$.
    
    We will solve this MSF problem with an MSF~$\phi$ as follows: since every pair $\hat\Psi(T_i), \hat\Psi(T_j)$ shares a nonempty edge set $\mathcal{E}(T_i, T_j)$, where $\mathcal{E}(\cdot, \cdot)$ is as defined in \cref{sec:typeb}, let~$\phi$ send $\frac{(\sigma(T_i) - \sigma(T_j))|\hat\Psi(T_i)||\hat\Psi(T_j)|}{|V(\mathfrak{a}_n)|}$ flow across the boundary $\mathcal{E}(\hat\Psi(T_i), \hat\Psi(T_j))$ for each ordered pair $(T_i, T_j)$ such that $\sigma(T_i) > \sigma(T_j)$. The resulting amount of flow is proportionally distributed among the classes according to their cardinalities, as required by the fact that $\delta$ is uniform. Furthermore, the function~$\phi$ as defined over the edges between $\hat\psi(T_i)$ subgraphs induces recursive distribution and concentration subproblems within each~$\hat\psi(T_i)$. That is, each~$\hat\psi(T_i)$ either receives or sends flow to or from each~$\hat\psi(T_j)$ for $j \neq i$. The proof of \cite[Lemma A.10]{eppstein2022improvedv1} showed that the incoming or outgoing flow across each boundary edge is sufficiently small that these subproblems are of the same form as the original problem and can be solved recursively with no increase in congestion. That is, the overall congestion is~$1$.
\end{proof}

\begin{lemma}
\label{lem:logtricktoplvl}
Consider a collection of MSFs of the form
$$(\sigma = |V(\mathfrak{b}_n)|, S = \Psi(D_i), F = V(\mathfrak{b}_n))$$ (so each class~$\Psi(D_i)$ must distribute its outbound flow throughout~$V(\mathfrak{b}_n)$).
This collection of MSFs can be solved while generating congestion at most~$O(\sqrt n)$.
\end{lemma}
\begin{proof}
    We follow the hierarchical grouping used in the proof of~\cite[Lemma 41]{eppstein2022improvedfull}: group the classes $\{\Psi(D_i)\}$ hierarchically into groups of 2, 4, 8, and so on, where each $k$-level group is a contiguous (with respect to the positions of the diagonals' endpoints) block $\{\Psi(D_{l+1}), \Psi(D_{l+2}), \dots, \Psi(D_{l+2^k})\}$ for some~$l$.

    At each level of the grouping, each pair of classes~$\Psi(D_i), \Psi(D_j)$ exchanges
    $$\frac{\sigma|\Psi(D_i)||\Psi(D_j)|}{2^k|\Psi(D_j)|}$$ flow across the matching $\mathcal{E}(D_i, D_j)$ since
    there are~$2^k$ classes in the group, all of equal size~$|\Psi(D_j)|$, each of which must receive its (equal) share of the $\sigma|\Psi(D_i)|$ flow sent out of~$\Psi(D_i)$. This produces congestion
    $$\frac{\sigma|\Psi(D_i)||\Psi(D_j)|}{|V(\mathfrak{b}_n)|2^k|\Psi(D_j)||\mathcal{E}(D_i, D_j)|} = \frac{|V(\mathfrak{b}_n)||\Psi(D_i)|}{|V(\mathfrak{b}_n)|2^k|\mathcal{E}(D_i, D_j)|} \leq \frac{O(|j-i|^{3/2})}{2^k} = O(\sqrt{2^k})$$
    across each edge in the matching $\mathcal{E}(D_i, D_j)$, where we have used the fact that~$D_i$ and~$D_j$ are in the same group. Furthermore, at the $k$-th grouping level, we have a single induced MSF within each~$\Psi(D_i)$ in which the surplus is bounded by~$\sqrt{2^k}|V(\mathfrak{b}_n)|$ and is uniform over each boundary set~$\partial_{D_j}(D_i)$. Treating each~$\partial_{D_j}(D_i)$ as a class $\Psi(T_j)$ as in \cref{sec:steps} permits us to apply \cref{lem:singlemsf}, allowing us to distribute the incoming flow throughout each~$\Psi(D_i)$ while producing congestion at most~$\sqrt{2^k}$. Summing over the~$k$ levels gives the~$O(\sqrt n)$ bound claimed.
\end{proof}

In light of the preceding lemmas, it suffices to perform the shuffling step as in \cref{sec:steps} and then to replace the concentration, transmission, and distribution steps with the flow constructed in the proof of \cref{lem:logtricktoplvl}.

\subsection{Mixing bound for~$\mathfrak{b}_n$}
Applying \eqref{eq:expmixing} to the~$O(\sqrt n \log n)$ expansion bound gives a mixing time of~$O(n^4\log^2n)$, proving \cref{thm:typeB_lower}. We now improve this to~$O(n^3 \log^3n)$. 

Lov\'{a}sz and Kannan gave a stronger version of \eqref{eq:expmixing} when one can show that small sets have large expansion. Given a graph~$G = (V, E)$, one can define the quantity
$$h(x) = \min_{S \subseteq V: |S|/|V| \leq x} |\partial S|/|S|.$$ Then the natural random walk on~$G$ has mixing time at most
\begin{equation}
\label{eq:avgcond}
    \tau \leq 32\int_{1/|V|}^{1/2} \frac{\Delta^2}{x(h(x))^2}dx.
\end{equation} 

Eppstein and Frishberg~\cite[Corollary 44]{eppstein2022improvedfull} showed how to partition the vertices of~$\mathfrak{a}_n$ into small sets with large expansion and apply \eqref{eq:avgcond}, obtaining the estimate 
\[    \int_{1/|V|}^{1/2} \frac{\Delta^2}{x(h(x))^2}dx = O(n^3 \log^3 n). 
\]
We retrace the main points of their analysis, beginning with the following lemma. 
\begin{lemma}[\cite{eppstein2022improvedfull}, Lemma~43]
\label{lem:efsmallpart}
For every integer $1 \leq k \leq n$, every triangulation~$t \in V(\mathfrak{a}_n)$ lies in the vertex set of some Cartesian product of associahedron graphs $\mathfrak{a}_{i_1}, \mathfrak{a}_{i_2}, \dots, \mathfrak{a}_{i_k}$, where $i_j \leq \frac{n}{2^{\lfloor \log_3 k\rfloor }}$ for all~$j$.
\end{lemma}
\cref{lem:efsmallpart} states that we can partition $V(\mathfrak{a}_n)$ hierarchically into a disjoint union of subsets, where each subset is the vertex set of a Cartesian product of associahedron graphs $\mathfrak{a}_{i_1}, \mathfrak{a}_{i_2}, \cdots, \mathfrak{a}_{i_k}$ with $i_j \leq \frac{n}{2^{\lfloor \log_3 k\rfloor }}$ for all $j$. Here, $k$ is an arbitrary parameter indicating how ``fine'' the partition is. A greater value of~$k$ means a greater number of (smaller) factor graphs in each Cartesian product.   

We sketch the hierarchical partitioning here. We need the following definition from \cite{eppstein2022improved}: a \dfn{central triangle} of the regular $(n+2)$-gon $P_{n+2}$ is a triangle $T$ (whose edges are edges and diagonals of $P_{n+2}$) such that $T$ contains the center of $P_{n+2}$. 

\begin{proof}[Proof sketch for \cref{lem:efsmallpart}]
For simplicity, assume~$n$ is odd. (This assumption does not affect the asymptotics.) Fix a parameter~$k$, and construct a partition of $V(\mathfrak a_n)$ as follows. If $k=1$, the partition is trivial and consists of the entire vertex set~$V(\mathfrak{a}_n)$. Otherwise, partition the triangulations in $V(\mathfrak{a}_n)$ into classes $\{\psi(T) : T \in \mathcal{T}\}$, where~$\mathcal{T}$ is the collection of all central triangles of $P_{n+2}$ and a triangulation~$t$ is in the class~$\psi(T)$ if and only if $t$ includes the central triangle~$T$. Observe that for each~$\psi(T)$, the triangle $T$ partitions the $(n+2)$-gon into three smaller polygons each with at most~$\frac{n}{2}$ vertices (one of these polygons may be degenerate and have only two sides). In arbitrary order, for each of these smaller polygons, recursively partition the triangulations in~$\psi(T)$ so that each triangulation~$t\in \psi(T)$ belongs to a class induced by the central triangle~$T'$ in the smaller polygon that~$t$ uses. Each such class is now a Cartesian product of up to $9 = 3^2$ smaller associahedron graphs, each defined over an $(i+1)$-gon with $i \leq \frac{n-2}{2^2} = \frac{n-2}{2}$.

Continue this process recursively until every triangulation is assigned to a Cartesian product of at least~$k$ smaller associahedron graphs. Once the process is stopped, each such product consists of smaller associahedron graphs defined over polygons with at most $1 + \max\{1, \frac{n-2}{2^{\lfloor \log_3 k\rfloor}}\}$ sides.
\end{proof}

Eppstein and Frishberg then applied the following lemma due to Chung and Tetali. 
\begin{lemma}[\cite{isocart}]
    \label{lem:gtexp}
    The expansion $h(G)$ of a Cartesian product~$G$ of graphs~$G_1, G_2, \dots, G_k$ is at least
    $$\frac{1}{2}\min_{1 \leq i \leq k} h(G_i).$$
\end{lemma}

\cref{lem:gtexp} implies the following corollary of \cref{lem:efsmallpart}. 
\begin{corollary}
\label{cor:smallsetsbigexp}
    Let $k\in[1,n]$. If $S \subseteq V(\mathfrak{a}_n)$ is such that $|S| \leq (C_{n/k})^k/2$, then 
    $$\frac{|\partial S|}{|S|} \geq h\left(\mathfrak{a}_{\frac{n}{2^{\lfloor \log_3 k \rfloor}}}\right) = \Omega\left(\frac{1}{\left(\frac{n}{2^{\lfloor \log_3 k \rfloor}}\right)^{3/2}\log \frac{n}{2^{\lfloor \log_3 k\rfloor}}}\right).$$
    \end{corollary}

Eppstein and Frishberg then obtained the desired bound by plugging \cref{cor:smallsetsbigexp} into \eqref{eq:avgcond} and applying a series of inequalities, which we include for completeness: 

\begin{align}
    \tau &\leq 32 \int_\frac{1}{C_n}^{1/2} \frac{\Delta^2}{(h(x))^2} dx \nonumber \\
    &= O(1)\sum_{k=1}^{n-1} \int_{\left(C_{\lfloor n/(k+1)\rfloor }\right)^{k+1}/C_n}^{\left(C_{\lfloor n/k\rfloor }\right)^k/C_n} \left(\frac{n}{2^{\lfloor \log_3 k\rfloor }}\right)^3 \log^2 \left(\frac{n}{2^{\lfloor \log_3 k\rfloor}}\right)\frac{dx}{x}\label{eq:subcorexp} \\
    &= O(n^3 \log^2 n)\sum_{k=1}^{n-1} \left(\frac{1}{2^{\lfloor \log_3 k\rfloor}}\right)^3 \int_{\left(C_{\lfloor n/(k+1)\rfloor}\right)^{k+1}/C_n}^{\left(C_{\lfloor n/k\rfloor}\right)^k/C_n}\frac{dx}{x} \label{eq:pulloutmainfacs} \\
    &= O(n^3 \log^2 n)  \sum_{k=1}^{n-1}\left(\frac{1}{2^{\lfloor \log_3 k\rfloor}}\right)^3\ln \frac{\left(C_{n/k}\right)^k}{\left(C_{n/(k+1)}\right)^{k+1}} \nonumber \\
    &= O(n^3 \log^2 n)  \sum_{k=1}^{n-1}\left(\frac{1}{2^{\lfloor \log_3 k\rfloor}}\right)^3\ln O(n^{3/2}) \label{eq:n32ratio} \\
    &= O(n^3 \log^3 n) \sum_{k=1}^{n-1}\left(\frac{1}{2^{\lfloor \log_3 k\rfloor}}\right)^3 \nonumber \\
    &= O(n^3 \log^3 n). \nonumber
\end{align}
In these computations, \eqref{eq:subcorexp} comes from applying \cref{cor:smallsetsbigexp} and breaking the integral into a sum of integrals, \eqref{eq:pulloutmainfacs} comes from pulling out the factors that do not depend on~$k$, and \eqref{eq:n32ratio} comes from the asymptotic behavior of Catalan numbers. 

Finally, we observe that this analysis carries through to~$\mathfrak{b}_n$

\begin{lemma}
    \label{lem:shaven}
    The $\Omega(\frac{1}{\sqrt n \log n})$ expansion bound in \cref{thm:typeB_lower} gives mixing time $O(n^3 \log^3 n)$.
\end{lemma}
\begin{proof}
    First, partition the triangulations of~$\mathfrak{b}_n$ according to the central diagonal as we have done elsewhere in this paper. Within each class, partition according to the central triangle on either of the two (symmetric) sides of the polygon. Repeat this process recursively. The proof follows from retracing the analysis in this section with the central-triangle partition of~\cite{eppstein2022improvedfull}. All of the inequalities hold as in~\cite{eppstein2022improvedfull}. 
\end{proof}
\end{document}